\documentclass[dvipsnames, 11pt]{article}

\usepackage[latin1, utf8]{inputenc}
\usepackage[T1]{fontenc}
\usepackage[english]{babel}
\usepackage{todonotes}
\usepackage{hyperref}
\usepackage{amsmath}
\usepackage{amssymb}
\usepackage{amsthm}
\usepackage{mathtools}
\usepackage{xkeyval}
\usepackage{stmaryrd}
\usepackage{authblk}

\usepackage[left=2.5cm,top=3cm,right=2.5cm,bottom=3cm,bindingoffset=0.5cm]{geometry}
\newcommand{\R}{\mathbb{R}}
\newcommand{\N}{\mathbb{N}}

\newcommand{\C}{\mathbb{C}}

\newcommand{\bx}{\mathbf x}
\newcommand{\Acal}{\mathcal A}
\newcommand{\E}{\mathbb E}

\DeclareMathOperator*{\argmax}{arg\,max}

\newcommand{\my}{\rho_{1/n}}
\newcommand{\leja}{\mathcal L_\tau(K)} 
\newcommand {\e}  {\varepsilon}
\newcommand{\vertiii}[1]{{\vert\kern-0.25ex\vert\kern-0.25ex\vert #1 
    \vert\kern-0.25ex\vert\kern-0.25ex\vert}}

\newtheorem{lmm}{Lemma}[section]
\newtheorem{thrm}{Theorem}[section]

\newtheorem{crllr}{Corollary}[section]
\newtheorem{rmrk}{Remark}[section]

\everymath{\displaystyle}
\title{Lebesgue constants of Leja sequences: a universal lower bound and improved upper bounds on the real line}
\author[a]{Camille Pouchol}
\affil[a]{Université Paris Cité, CNRS, MAP5, F-75006 Paris, France. \footnote{Email address: camille.pouchol@u-pariscite.fr.}}
\date{\empty}

\begin{document}
\maketitle
\begin{abstract}
We prove a general lower bound for the Lebesgue constants $\Lambda_n$ of
Leja sequences for every nonpolar compact set \(K\subset\mathbb C\): $\textstyle \limsup_{n \to+\infty} \Lambda_n \, \sqrt{n^{-1} W_n(K)}\geq  \tfrac{1}{\sqrt{2}}$, where \(W_n(K)\) denotes the Chebyshev--Widom factor. In particular, bounded Widom factors imply a universal square-root obstruction.
We also establish improved polynomial upper bounds on compact subsets $K$ of the real line, building on the approach of Andrievskii and Nazarov. Our first ingredient is a refinement of their Key Lemma: we lower its exponent from $9/8 = 1.125$ to a near-optimal $\gamma^\star \approx 0.6304$. Our second ingredient is a localisation of the separation argument: instead of using the global minimum spacing among the first $n$ Leja points, we retain the local Green-function
scale.  For finite unions of intervals, a packing argument and sharp estimates on
the latter function near and away from the endpoints yield  $\Lambda_n = O(n^{1 + \gamma^\star})$, uniformly over Leja sequences. This improves the previously known best exponent $13/4 = 3.25$ to around
$1.6304$ for such sets (including the benchmark case $K = [-1,1]$), hence substantially narrowing the gap with the near-linear growth observed numerically. Corresponding polynomial bounds for $\tau$-Leja sequences are derived. 
\end{abstract}

\section{Introduction}

Let $K$ be a nonpolar compact subset of $\C$. We are interested in Leja sequences associated with~$K$, i.e., sequences $(x_n)_{n \in \N} \in K^\N$ satisfying
\begin{equation}
\label{genuine_leja} 
\forall n \in \N^*, \quad x_n \in \argmax_{x \in K} \prod_{k=0}^{n-1}|x - x_k|.
\end{equation}

Leja sequences combine a strong practical advantage, that of nestedness, with a still poorly understood interpolation stability, which is governed by the growth of the associated Lebesgue constants. Even for the benchmark interval $K = [-1,1]$, numerical evidence suggests growth close to linear, whereas the best available uniform polynomial estimate for Leja sequences is~$O(n^{13/4})$. The purpose of this paper is to reduce this gap.

\paragraph{Definitions.}
Following~\cite{AN} (see also~\cite{PseudoLeja2012} for even more general relaxations), we define the set of $\tau$-Leja sequences associated with $K$, where we relax the maximisation assumption above to maximisation up to a fixed multiplicative error $\tau \in (0,1]$:
\begin{equation}
\label{tau_leja}
\mathcal L_\tau(K) := \bigg\{\bx = (x_n)_{n \in \N} \in K^\N \, \bigg| \, \; \forall n \in \N^*, \; \prod_{k=0}^{n-1}|x_n- x_k|   \geq \tau \sup_{x \in K} \prod_{k=0}^{n-1}|x- x_k|\bigg\}.
\end{equation}
For $\tau = 1$, $\mathcal L(K):= \mathcal L_1(K)$ recovers  genuine Leja sequences defined by~\eqref{genuine_leja}. Note that, already for $\tau = 1$, $\mathcal L_1(K)$ is infinite due to the freedom of $x_0 \in K$, but also because there can be several maximisers for the product underlying~\eqref{genuine_leja}.

Our interest is in the quality of polynomial interpolation based on $\bx \in \leja$. If $L_n^\bx(f)$ denotes the unique interpolating polynomial of degree $n-1$ of a continuous function $f \in \mathcal C(K;\C)$ at the points $x_0, \ldots, x_{n-1}$, the relevant stability quantity is the Lebesgue constant, namely the operator norm of $f \mapsto L_n^\bx(f)$, with $\mathcal C(K;\C)$ endowed with the supremum norm $\|\cdot\|_K$. Indeed, the following inequality holds
\begin{equation}
\label{bound_lebesgue}
\|L_n^\bx(f) - f\|_K \leq (1 + \Lambda_n^\bx) \inf_{\mathrm{deg}(P) \leq n-1} \|P - f\|_K.
\end{equation}
As a result, depending on the additional regularity $f$ may have, the right-hand side converges to~$0$ provided that $(\Lambda_n^\bx)_{n \in \N^*}$ grows sufficiently slowly.

The Lebesgue constant is more explicitly given by  \[\Lambda_n^\bx := \sup_{x \in K} \sum_{k=0}^{n-1} |\ell^\bx_{k,n}(x)|, \qquad  \ell_{k,n}^\bx(x) := \prod_{\substack{0 \leq j  \leq n-1 \\ j \neq k}} \frac{x - x_j}{x_k -x_j}. \]

 \paragraph{State of the art.} 

Leja sequences were introduced in~\cite{Leja1957} and are attractive for interpolation because, unlike Fekete or Chebyshev points, they form nested sequences. Their interpolation properties are well understood only in a few special situations: precise polynomial bounds are available for the unit disk~\cite{Calvi_UnitDisk_2011, Ounaies2017} and for real projections of Leja sequences on the disk~\cite{CalviPhung2012, Lebesgue_Leja_Disk2013}, which differ from genuine Leja sequences on an interval. 

For general sets, Taylor and Totik established non-exponential growth of the Lebesgue constants for a broad class of planar compact sets~\cite{Totik2010}, and Totik later proved subexponential growth for arbitrary nonpolar compact subsets of $\mathbb C$~\cite{Totik2023}. For Leja interpolation, we are not aware of polynomial lower bounds of positive order beyond a few special settings such as the unit disk.

On the real line, Andrievskii and Nazarov obtained the first explicit polynomial bounds for a large class of compact sets~\cite{AN}. In particular, for genuine Leja sequences on $K=[-1,1]$, their estimate gives $\Lambda_n=O(n^{13/4})$, whereas numerical evidence suggests a growth close to $O(n)$~\cite{Reichel1990}. The remaining gap between these two rates is the starting point of the present work.

\subsection{Main results}

\paragraph{General lower bound.}
To the best of our knowledge, however, no lower bound of positive polynomial order has been established for Leja sequences on $[-1,1]$, let alone on general nonpolar compact subsets of $\C$.
We derive a new lower bound based on the Widom factors of $K$.

Let $c_K$ be the logarithmic capacity of $K$, and for $n \geq 1$ define the Widom factor through
\[
    t_n(K)
    :=
    \inf\left\{
        \|Q\|_K\,|\,
        Q \text{ is monic of degree }n
    \right\}, \quad 
    W_n(K):=\frac{t_n(K)}{c_K^n}.
\]
Note that \(W_n(K)\geq1\).

\begin{thrm}
\label{lower_bound}
Assume that $K$ is a nonpolar compact subset of $\C$.  Then
\[\forall \tau \in (0,1], \; \forall  \bx \in \leja, \quad \limsup_{n\to+\infty} \; \Lambda_n^\bx \sqrt{\frac{W_n(K)}{n}} \geq \frac1{\sqrt2}. \]
If $W_n(K) = o(n)$, the lower bound $\tfrac1{\sqrt2}$ can be replaced by $1$. 
\end{thrm}
For finite unions of intervals, $W_n(K) = O(1)~$\cite[Theorem~2]{Andrievskii2017}, hence the above theorem proves that any polynomial upper bound valid for any $\tau$-Leja sequences on finite unions of intervals (including $[-1,1]$) must have exponent at least \(1/2\). More generally, Theorem 1 in \cite{Andrievskii2017} states that $W_n(K) = O(\ln(n))$ for all compact connected sets with positive diameter.
In particular, this rules out an upper bound
$\Lambda_n^\bx=O(n^\alpha)$ for any $\alpha<1/2$, within this very large class.

\paragraph{Upper bounds on the real line.} We now turn to upper bounds on the real line. By tightening and
localising the approach of~\cite{AN}, we substantially reduce the gap
between the known $O(n^{13/4})$ estimate and the near-linear growth
suggested numerically.

The first major improvement comes from obtaining an improved exponent within the \textit{Key Lemma} of~\cite{AN} (Lemma 2 therein), from $9/8 = 1.125$ to $\gamma^\star \approx 0.6304$.
The second, more conceptual improvement, comes from noticing that the estimate of Andrievskii and Nazarov concerning the $k$-th fundamental polynomial $\ell_{k,n}^\bx$ involves the distance from $x_k$ to the subsequent Leja points. In~\cite{AN}, this dependence is
eventually discarded by replacing it with the global minimum spacing. 
We retain this local information and show that the relevant
separation scale is instead $\rho_{1/n,K}(x_k)$, the distance from $x_k$ to the
level set $\{g_K=1/n\}$ where $g_K$ is the Green function of $K$. 

We thereby derive an improved polynomial estimate for the Lebesgue constant of $\tau$-Leja sequences for any (regular compact) \textit{uniformly perfect} subset of $\R$. In this case, there exists $\beta>0$ such that $\Lambda_n^\bx = O(n^{1+ \gamma^\star /\beta})$, for $\tau =1$. For a finite union of intervals, one may take $\beta = \tfrac{1}{2}$, which leads to  $\Lambda_n^\bx = O(n^{1+ 2 \gamma^\star})$, already yielding an improvement from $3.25$ to $1+2 \gamma^\star \approx 2.2609$.

The local improvement further reduces the exponent in the specific case of finite unions of intervals. In such a case, for genuine Leja sequences $\tau=1$ on a finite union of intervals, the previously known exponent $13/4=3.25$ is reduced to
\[\Lambda_n^\bx = O(n^{1+
\gamma^\star}), \qquad 1+ \gamma^\star \approx 1.6304.\]

\medskip

\bigskip 

\textit{Upper bounds.} Let us now give our main results in full detail, which rely on the following two numerical constants: 
\begin{itemize}
\item \[\text{$\lambda$ is the positive solution of $e^{e^{\lambda}}(e^{\lambda}-1) = 1$},\] which satisfies $\lambda \approx 0.24565978$,
\item 
\begin{equation} 
\label{gamma_star}
\gamma^\star := \sup_{0<q<p<1} R(p,q), \qquad R(p,q)
:=
\frac{
- q \ln(q) -(1-q)\ln(1-q)-(1-p)\ln \left(\frac{p}{p-q}\right)
}{
\ln\!\left(\frac{1-q}{1-p}\right)},
\end{equation}
which satisfies $\gamma^\star \approx 0.6304$.\footnote{With the change of variables
$u=\tfrac{p-q}{1-q}$, which maps the domain $0<q<p<1$ onto $(0,1)^2$, a numerical
maximisation of the resulting two-variable function gives $\gamma^\star \approx 0.6304432430$, with a numerical maximiser at $(p,q)\approx(0.674028143,\,0.436732437).$}
\end{itemize}
For $\tau \in (0,1]$, also define 
\[\gamma_\tau:=\gamma^\star  + 2 \lambda^{-1} \ln(\tau^{-1}).\]

To give the result, we shall need to set some notations. For distinct real numbers $(x_0, \ldots, x_q) \in \R^{q+1}$, we let
\[
I_\tau(x_0,\ldots,x_q)
:=
\min_{0=n_0<\cdots<n_m=q} \;
\tau^{-m}
\frac{
\displaystyle
\prod_{\ell=0}^{m-1}
\prod_{j=n_\ell}^{n_{\ell+1}-1}
|x_{n_{\ell+1}}-x_j|
}{
\displaystyle
\prod_{j=1}^{q}|x_0-x_j|
}.
\]
 The interest of this definition is the following bound, for $\bx \in \leja$
\begin{equation}
\label{interest_Itau}
\forall k \in \{0, \ldots, n-2\},\; \forall x \in K \setminus \{x_k, \ldots, x_{n-1}\},  \quad |\ell_{k,n}^\bx(x)| \leq I_\tau(x_k, \ldots, x_{n-1}, x).
\end{equation}
Such a bound is obtained by repeatedly using the definition of $\tau$-Leja sequences.

Our first result is a tighter estimate for the Key Lemma of~\cite{AN}. 
\begin{lmm}
\label{improved_AN}
For any distinct real numbers $(x_0, \ldots, x_q) \in \R^{q+1}$ (with $q \geq 2$) satisfying 
\begin{equation}
\label{hyp_xi} 
 |x_j - x_0|\le D,\quad 1\le j\le q, \quad \text{and} \quad 
 |x_j-x_0|\ge\Delta,\quad1\le j\le q-1,
\end{equation}
there holds
\[I_\tau(x_0, \ldots, x_q) \leq 2 \tau^{-2} \, \Big(\frac{D}{\Delta}\Big)^{\gamma_\tau}.\]
\end{lmm}
\begin{rmrk}
For $\tau = 1$, we have $\gamma_\tau = \gamma^\star$, and this estimate happens to be nearly optimal. Indeed, consider the choice $x_0 = 0$, $x_j = (-2)^{j-1}$ for $1 \leq j \leq q$, which is associated to $D = 2^{q-1}$ and $\Delta = 1$. In this case, $I_1(x_0, \ldots,x_q) = (\tfrac{3}{2})^{q-1} =  (\tfrac{D}{\Delta})^{\log_2(3/2)}$. As a result, the best possible exponent $\gamma$ (for $\tau = 1$) in the statement above is bound to satisfy $\gamma \geq \log_2(3/2) \approx 0.5850$. This is why we deem the value $\gamma^\star \approx 0.6304$ to be rather sharp.
\end{rmrk}

In particular, the arguments of~\cite{AN} then show polynomial growth of the Lebesgue constants for $\tau$-Leja sequences on uniformly perfect sets, and this new exponent improves it. Denoting $G_K : \delta \mapsto \textstyle \max \{g_K(z)\, |\, \mathrm{dist}(z,K) \leq 2 \delta\}$, for such sets there exists $\beta > 0$ such that $G_K(\delta) = O(\delta^\beta)$ as $\delta \to 0$.
\begin{crllr}
Assume that $K \subset \R$ is a uniformly perfect set. Let $\beta>0$ be such
that $G_K(\delta)=O(\delta^\beta)$.
Then, uniformly over $\tau$-Leja sequences $\bx \in \leja$, we have $\Lambda_n^\bx = O(n^{1 + \gamma_\tau/\beta})$.
\end{crllr}

Our second result is a localised version of the argument in~\cite{AN}, which leads to a further improvement in the case of finite unions of intervals. 
\begin{thrm}
\label{main_thm}
Assume that $K$ is a finite union of nondegenerate intervals. Then, uniformly over $\tau$-Leja sequences $\bx \in \leja$,
\[ \Lambda_n^\bx  = \begin{cases} O(n^{1+\gamma_\tau}) & \text{ if } \gamma_\tau<1  \\
O(n^2 \ln(n)) &\text{ if }  \gamma_\tau = 1 \\
O(n^{2\gamma_\tau})  &\text{ if }  \gamma_\tau > 1
 \end{cases}.\]
\end{thrm}
In the above result, the constant associated with $O$ is independent of $\bx \in \leja$, but it depends both on $K$ and $\tau$.
Given the expression of $\gamma_\tau$, one can check that the threshold corresponding to $\gamma_\tau = 1$ is given by $\tau  = \exp(-\tfrac{\lambda}{2}(1-\gamma^\star)) \approx 0.9556$. For the important value $\tau = 1$, we are in the regime $\gamma_\tau<1$: the estimate is $O(n^{1+\gamma^\star})$ for genuine Leja sequences, as already mentioned.

\paragraph{Strategy of proof.}
 For $k \leq n-2$, and $\bx \in \leja$, let \[\e_{k,n}^\bx := \inf_{k <j< n} |x_k - x_j|.\]
Combining Lemma~\ref{improved_AN} with the bound~\eqref{interest_Itau}, and letting $d_K := \mathrm{diam}(K)$, we obtain for all $0 < \tau \leq 1$ and $\bx \in \leja$
\begin{equation}
\label{bound_AN_eq}
\forall k \in \{0, \ldots, n-2\}, \quad \|\ell_{k,n}^\bx\|_K  \leq C \tau^{-2}\, \Big(\frac{d_K}{\e_{k,n}^\bx}\Big)^{\gamma_\tau}, 
\end{equation}

In~\cite{AN}, this estimate is eventually used in a weaker form, where $\e_{k,n}^\bx$ is replaced with the overall minimal spacing up to index $n-1$, namely $\e_n^\bx := \textstyle \inf_{0 \leq k<j <n} |x_k - x_j|$.
For finite unions of intervals, it is well known that the above spacing satisfies $\e^\bx_n \gtrsim n^{-2}$. Then, combining the crude bound
\begin{equation}
\label{crude_bound} 
\Lambda_n^\bx \leq n \max_{0 \leq k \leq n-1} \|\ell_{k,n}^\bx\|_K,
\end{equation}
with these two estimates leads, uniformly in $\bx \in \leja$, to $\Lambda_n^\bx = n \, O\big(n^{2\gamma_\tau}\big) = O\big(n^{1+ 2\gamma_\tau}\big)$.

This is the estimate produced by the global-separation argument. With the exponent \(9/8\) of~\cite{AN}, it yields the previously best known bound $\Lambda_n^\bx = O(n^{13/4})$, which corresponds to the exponent $1+2 \times (9/8)$ when one uses the suboptimal $9/8$ rather than $\gamma^\star$.

The dependence of~\eqref{bound_AN_eq} with respect to the index $k$ is lost, which suggests a local approach. 
Our second main improvement comes from 
\begin{itemize}
\item[(i)] deriving a local lower bound $\e_{k,n}^\bx \gtrsim \rho_{1/n,K}(x_k)$, which can be seen as a local version of Lemma~1 in~\cite{AN}. This is the content of Lemma~\ref{est_spacing}.
\item[(ii)] in the case of a finite union of nondegenerate intervals, showing that $\textstyle \sum_{k=0}^{n-1} \rho_{1/n,K}^{-\alpha}(x_k)  \lesssim \int_K \rho_{1/n,K}^{-(\alpha+1)}(x) \,dx$ for all $\alpha >0$. This is based on Lemma~\ref{sum_to_int}.
\item[(iii)] providing an estimate $\textstyle \int_K \rho_{1/n,K}^{-(\alpha+1)}(x) \,dx$, which is $O(n^{1+\alpha})$ for $0< \alpha<1$ and $O(n^{2\alpha})$ for $\alpha>1$, thanks to sharp estimates for $\rho_{1/n,K}$. This is the content of Lemma~\ref{estimate_int} and Lemma~\ref{estimate_rho}, respectively.
\end{itemize}

If $\rho_{1/n, K}$ were of the order $n^{-2}$ on the whole of $K$, nothing would be gained by replacing the crude bound~\eqref{crude_bound} by this sharper bound. Since $\rho_{1/n, K}$ in fact scales like $n^{-1}$ away from the endpoints and only deteriorates to \(n^{-2}\) near the endpoints, exploiting this local scale leads to a strictly sharper summation estimate.


\section{Universal lower bound}
Here, we prove Theorem~\ref{lower_bound}. Recall that $K$ is any nonpolar compact subset of $\C$.

Fix $\bx \in \mathcal L_\tau(K)$. For \(n\geq1\), set
\[
    \omega_n(z):=\prod_{k=0}^{n-1}(z-x_k), \qquad 
    D_n:=|\omega_n(x_n)|.
\]
The \(\tau\)-Leja property reads $D_n\geq\tau \|\omega_n\|_K$.

Since every monic polynomial \(P\) of degree \(n\) satisfies $\|P\|_K\geq c_K^n$,
we have $t_n(K)\geq c_K^n$ and $W_n(K)\geq1$.
In particular, since \(\omega_n\) is monic, $\|\omega_n\|_K\geq t_n(K)\geq c_K^n$. 
Combining this with the $\tau$-Leja property yields $D_n\geq\tau c_K^n$. 
We therefore introduce the normalised product $d_n:=\tfrac{D_n}{c_K^n}$,  
which satisfies
\begin{equation}
\forall n \geq 1, \quad 
    d_n\geq\tau.
    \label{eq:dn-lower-bound}
\end{equation}

\medskip
\noindent
\textit{First lower bound.}
Let \(Q_n\) be a monic Chebyshev polynomial of degree \(n\) for
\(K\), so that $\|Q_n\|_K=t_n(K)$.
We now use the interpolation operator \(L_n^\bx\) at the
nodes \(x_0,\ldots,x_{n-1}\): since $Q_n-L_n^\bx(Q_n)$
is monic of degree \(n\) and vanishes at
\(x_0,\ldots,x_{n-1}\), we have
\[
    Q_n-L_n^\bx(Q_n)=\omega_n.
\]
Evaluating this identity at \(x_n\) yields
\[
    D_n
    =|\omega_n(x_n)|
    \leq |Q_n(x_n)|+
           |L_n^\bx(Q_n)(x_n)| 
    \leq \bigl(1+\Lambda_n^\bx\bigr)t_n(K).
\]
Since $D_n=c_K^n d_n$ and $t_n(K)=c_K^nW_n(K)$,
we obtain
\begin{equation}
    \Lambda_n^\bx
    \geq
    \frac{d_n}{W_n(K)}-1.
    \label{eq:first-general-lower-bound}
\end{equation}

\medskip
\noindent
\textit{Second lower bound.}
For \(0\leq k \leq n-1\), we have
\[
    \ell_{k,n}^\bx(z)
    =
    \frac{\omega_n(z)}
         {(z-x_k)\omega_n'(x_k)}, 
\]
so that evaluating at \(x_n\), we find
\[
    |\ell_{k,n}^\bx(x_n)|
    =
    \frac{D_n}
         {|x_n-x_k|\,|\omega_n'(x_k)|}.
\]
Taking the product over \(k=0,\ldots,n-1\) and using $\textstyle \prod_{k=0}^{n-1}|x_n-x_k|=D_n$,
we obtain
\begin{equation}
    \prod_{k=0}^{n-1}|\ell_{k,n}^\bx(x_n)|
    =
    \frac{D_n^{n-1}}
         {\displaystyle\prod_{k=0}^{n-1}|\omega_n'(x_k)|}.
    \label{eq:lagrange-general-product-1}
\end{equation}
Moreover,
\[
    \prod_{k=0}^{n-1}|\omega_n'(x_k)|
   =
    \prod_{0\leq j<k<n}|x_k-x_j|^2 
    =
    \prod_{k=1}^{n-1}
    \left(
        \prod_{j=0}^{k-1}|x_k-x_j|
    \right)^2 
    =
    \prod_{k=1}^{n-1}D_k^2.
\]
Therefore,
\[
    \prod_{k=0}^{n-1}|\ell_{k,n}^\bx(x_n)|
    =
    \frac{D_n^{n-1}}
         {\displaystyle\prod_{k=1}^{n-1}D_k^2}.
\]
Substituting \(D_k=c_K^k d_k\), all powers of \(c_K\) cancel, and
we obtain
\begin{equation}
    \prod_{k=0}^{n-1}|\ell_{k,n}^\bx(x_n)|
    =
    \frac{d_n^{n-1}}
         {\displaystyle\prod_{k=1}^{n-1}d_k^2}.
    \label{eq:lagrange-general-product-2}
\end{equation}

By the arithmetic-geometric mean inequality,
\[
    \Lambda_n^\bx
 \geq
    \sum_{k=0}^{n-1}|\ell_{k,n}^\bx(x_n)|
 \geq
    n\left(
        \prod_{k=0}^{n-1}|\ell_{k,n}^\bx(x_n)|
    \right)^{1/n}.
\]
It follows from \eqref{eq:lagrange-general-product-2} that
\begin{equation}
    \Lambda_n^\bx
    \geq
    n
    \frac{d_n^{(n-1)/n}}
         {\left(\displaystyle\prod_{k=1}^{n-1}d_k\right)^{2/n}}.
    \label{eq:second-general-lower-bound}
\end{equation}

We now distinguish two cases.
Suppose first that there are infinitely many indices \(n\) (which we call \textit{record} indices) such that 
$d_n= \textstyle \max_{1\leq k\leq n}d_k$.
Let $n$ be a fixed record index. Then $\textstyle 
    \prod_{k=1}^{n-1}d_k\leq d_n^{n-1}$, 
and hence \eqref{eq:second-general-lower-bound} gives
$ \Lambda_n^\bx
    \geq
    n \, d_n^{-(n-1)/n}$.
Since $d_n\ge\tau$, we obtain
$\Lambda_n^\bx
\ge
\tau^{1/n}\tfrac{n}{d_n}$. On the other hand, multiplying by $d_n\le W_n(K)\bigl(\Lambda_n^\bx+1\bigr)$ given by~\eqref{eq:first-general-lower-bound} gives
\[
\Lambda_n^\bx(\Lambda_n^\bx+1)
\ge
\tau^{1/n}\frac{n}{W_n(K)}.
\]

Since every Lebesgue constant satisfies $\Lambda_n^\bx\geq1$, it follows that \[ \left(\Lambda_n^\bx\right)^2 \frac{W_n(K)}{n} \geq \tau^{1/n} \frac{\Lambda_n^\bx}{\Lambda_n^\bx+1} \geq \frac{\tau^{1/n}}2. \] Thus, at every record index, \[ \Lambda_n^\bx \sqrt{\frac{W_n(K)}{n}} \geq \frac{\tau^{1/(2n)}}{\sqrt2}. \] Since there are infinitely many such indices and $\tau^{1/(2n)}\to1$, we conclude that \[ \limsup_{n\to+\infty} \; \Lambda_n^\bx \sqrt{\frac{W_n(K)}{n}} \geq\frac1{\sqrt2}. \]

Moreover, if $W_n(K)=o(n)$, then
\[
\tau^{1/n}\frac{n}{W_n(K)}\longrightarrow+\infty.
\]
Hence $\Lambda_n^\bx$ tends to $+\infty$ along the record indices, and therefore
\[
\frac{\Lambda_n^\bx}{\Lambda_n^\bx+1}\longrightarrow1.
\]
Returning to the preceding inequality yields
\[
\limsup_{n\to\infty}\;
\Lambda_n^\bx\sqrt{\frac{W_n(K)}{n}}
\ge1.
\]

It remains to consider the case in which there are only finitely many
indices \(n\) satisfying $\textstyle
    d_n=\max_{1\leq k\leq n}d_k$. 
Then the sequence \((d_n)\) is bounded above by some $R \geq 1$.
Using \eqref{eq:dn-lower-bound} and~\eqref{eq:second-general-lower-bound}, we obtain
\[ \Lambda_n^\bx \geq n\frac{\tau^{(n-1)/n}}{R^{2(n-1)/n}} \geq \frac{\tau}{R^2}\,n. \] Since $W_n(K)\geq1$, this gives \[ \Lambda_n^\bx \sqrt{\frac{W_n(K)}{n}} \geq \frac{\tau}{R^2}\sqrt n \longrightarrow+\infty. \]

In all cases, the announced bound is thus satisfied.
\qed

\section{Refinement for the Key Lemma of~\cite{AN}}

\subsection{A probabilistic view on the proof of~\cite{AN}}
For convenience, we shall reverse the order of $(x_0, \ldots, x_q)$. In this section, we thus work with a fixed set of distinct real numbers $(x_0, \ldots, x_q) \in \R^{q+1}$, and our reordering means that the assumption~\eqref{hyp_xi} becomes
\begin{equation}
\label{hyp_xi_bis} 
 |x_j - x_q|\le D,\quad 0\le j\le q-1, \quad \text{and} \quad 
 |x_j-x_q|\ge\Delta,\quad1\le j\le q-1,
\end{equation}
and we now must bound $I_\tau(x_q, \ldots, x_0)$ under this assumption.

 As already noticed in~\cite{AN}, $I_\tau(x_q, \ldots, x_0)$ rewrites as follows:
\[I_\tau(x_q, \ldots, x_0) = \min_{X \in \Acal} Q_\tau(X), \qquad Q_\tau(X):=  \tau^{-m(X)} \, 
 \frac{|X_{q-1}-x_q|}{|x_0-x_q|}
 \prod_{k=0}^{q-2}
 \frac{|X_k-x_{k+1}|}{|x_{k+1}-x_q|},\]
 where
\[\Acal =  \Acal(x_0, \ldots, x_q) := \big\{X = (X_0,\ldots, X_{q-1}) \in \R^q \, |\,  X_0=x_0,
 \;
X_{k+1}\in\{X_k,x_{k+1}\},
 \, 0\le k\le q-2\big\}.\]
 and for $X \in \Acal$,  $m(X):=  \#\{k \in \{0, \ldots, q-2\} \, | \, X_k \neq X_{k+1} \} + 1$.

 To obtain a good estimate for $I_\tau(x_q, \ldots, x_0)$, we must thus find $X \in \Acal$ that makes $Q_\tau(X)$ as small as possible. We will exactly follow the strategy $X$ devised in~\cite{AN}, but  will study the induced $Q_\tau(X)$ in a sharper way. Their multiplicative averaging argument admits the following probabilistic reformulation: one builds a random vector $(X_0, \ldots, X_{q-1})$ taking values in $\Acal$ almost surely, and provides estimates for $\E[\ln Q_\tau(X)]$, which upper bounds $\ln(I_\tau(x_q, \ldots, x_0))$.

Upon translating and dilating, we may assume $x_q = 0$ and $\Delta = 1$, so that~\eqref{hyp_xi_bis} becomes, with $R := D/\Delta$,
\begin{equation}
\label{hyp_xi_bis_bis}
 |x_j|\le R,\quad 0\le j\le q-1, \quad \text{and} \quad
 |x_j|\ge1,\quad1\le j\le q-1.
 \end{equation}
 Consequently, we now aim at controlling $\ln(I_\tau(0, x_{q-1}, \ldots, x_1, x_0))$ linearly as a function of $\ln(R)$.

For $a,b>0$ and $0<c<a$, set
\[
C(a,b;c)
:=
\frac{b}{a+b}\ln\frac{a-c}{c}
+
\frac{a}{a+b}\ln\frac{b+c}{c}, \qquad
T(a,b)
:=
\frac{b}{a+b}\ln a
+
\frac{a}{a+b}\ln b.
\]

As we shall see, $C(a,b;c)$ is involved in the cost of one term in the sum defining $\ln Q_\tau(X)$ for the chosen random vector $X$, which is built to make these costs telescope. This is why we make the following definition: we say that $\Phi : (0,+\infty)^2\to\mathbb R$ is a \textit{potential} if it is symmetric and satisfies \begin{equation}
\label{telescope}
C(a,b;c)
\le
\Phi(a,b)-\Phi(c,b),
\qquad
a,b>0,\quad 0<c<a.
\end{equation}

As the following Lemma will reveal, for a given potential $\Phi$, the relevant function ends up being 
\[
G_\Phi(a,b,a',b')
:=
\ln\frac{a+b}{ab}
+
\Phi(a,b)
-
\Phi(a',b')
+
T(a',b').
\]

The Lemma involving $G_\Phi$ is a convenient rephrasing of the proof of Key Lemma in~\cite{AN}. For pedagogical purposes, we provide a proof that insists on the probabilistic approach by making use of random variables, and highlights why the aforementioned functions appear naturally. We also believe this form might be of interest for possible generalisations or further improvements of our estimates.

\begin{lmm}
\label{lem:potential-reduction}
Let $\Phi:(0,+\infty)^2\to\mathbb R$ be a potential. 
Then if $(x_0, \ldots, x_{q-1}) \in \R^{q}$ satisfies~\eqref{hyp_xi_bis_bis}, there exists $(a,b,a',b') \in E_R$, where  \[E_R :=\left\{(a,b,a',b') \in (0,R]^4 \, |\, 1 \leq a' \leq a, \; \big(b' = b<1 \text{ or } 1 \leq b' \leq b\big)\right\},\]
such that
\[
\ln I_\tau(0, x_{q-1}, \ldots,x_0) 
\leq
G_\Phi(a,b,a',b') + 2 \ln(\tau^{-1}) + 2 \lambda^{-1} \ln(\tau^{-1}) \ln(R).
\]
\end{lmm}
\begin{proof} 
For a random vector $X$ taking values in $\Acal$ almost surely, we have by definition
\[\E[\ln(Q_\tau(X))] = \E[m(X)]  \ln(\tau^{-1})  +  \E\left[\ln \frac{|X_{q-1}|}{|x_0|}\right] + \sum_{k=0}^{q-2} \E\left[\ln \frac{|X_k-x_{k+1}|}{|x_{k+1}|}\right].\]
Following~\cite{AN}, we focus on the case where, for all  $0 \leq k\leq q-1$, $X_k$ follows a discrete law with at most two atoms. 
Let us set \[c_k :=  \E\left[\ln\frac{|X_k-x_{k+1}|}{|x_{k+1}|}\right], \quad k \in \{0, \ldots, q-2\},\]
which can be seen as the cost of the $k$th term. We recursively define the random vector $X$. Upon multiplying all the $x_i$'s by $-1$, we may assume that $x_0>0$. 

\textit{Positive case.} 
If $x_0, x_1, \ldots, x_{q-1}$ are all positive, we use the deterministic strategy
\[
X_{k+1}=
\begin{cases}
X_k,&x_{k+1}\ge e^{-\lambda}X_k,\\
x_{k+1},&x_{k+1}<e^{-\lambda}X_k.
\end{cases}
\]
In either case,
\[
c_k\le \ln\frac{X_k}{X_{k+1}},
\]
and hence
\[
\sum_{k=0}^{q-2}c_k\le \ln\frac{x_0}{X_{q-1}},
\qquad
\mathbb E[\ln Q_1(X)]\le0.
\]
Moreover, each update decreases the reference point by at least a factor
$e^\lambda$, so that $m(X)\le 1+\lambda^{-1}\ln R \leq 2+ 2 \lambda^{-1} \ln R$. Finally, since $G_\Phi(1,1,1,1) = \ln(2)$ and $(1,1,1,1) \in E_R$, the bound is valid in this case.

%

\textit{General case.} Let us now drop the positivity assumption, and let $k^\star := \min\{k \in \{1, \ldots, q-1\}\, |\,x_k < 0\}$. For $k \leq k^\star-2$, we choose $X_k$ deterministically as above, leading to the bound
\[\sum_{k=0}^{k^\star-2} c_k \leq \ln \frac{x_0}{X_{k^\star-1}}.\]
We then set $b_{k^\star} = X_{k^\star-1}$, $a_{k^\star} = -x_{k^\star}$, so that $c_{k^\star -1} = \ln(\tfrac{a_{k^\star} + b_{k^\star}}{a_{k^\star}})$.

For $k \geq k^\star$, we then only choose random variables whose law belongs to the family
\[\mu_{a,b} = \tfrac{b}{a+b} \delta_{-a} +  \tfrac{a}{a+b} \delta_b, \qquad a,b>0.\]
Note that $X \sim \mu_{a,b}$ satisfies $\E[X] = 0$. Choosing $\mu_{a,b}$ to be centered is instrumental: if $x \in \R$ and $a,b>0$ are such that $x \notin e^{-\lambda} (-a,b)$, then for $X \sim \mu_{a,b}$,  the Elementary Inequality 1 in~\cite{AN} exactly reads
\[ \E\left[\ln\frac{|X-x|}{|x|}\right] \leq 0.\]

Let us now assume $X_k \sim \mu_{a_k, b_k}$ for $k \geq k^\star$, where the sequences $a_k, b_k > 0$ are recursively defined for $k^\star \leq k \leq q-2$ by the recursion
\begin{itemize}
\item[(i)] if $x_{k+1} \notin  e^{-\lambda} (-a_k,b_k)$, we let $(a_{k+1}, b_{k+1}) = (a_k, b_k)$,	
\item[(ii)] if $- e^{-\lambda}a_{k} < x_{k+1} <0$, we let $(a_{k+1}, b_{k+1}) = (-x_{k+1},b_k)$,
\item[(iii)] if $0<x_{k+1} <e^{-\lambda} b_{k} $, we let $(a_{k+1}, b_{k+1}) = (a_k, x_{k+1})$.
\end{itemize} 
Let us first check that this choice can be made admissible, i.e., that we can choose couplings with the above marginals so that $(X_0, \ldots, X_{q-1}) \in \Acal$ almost surely. 
In cases (ii) and (iii), the two consecutive barycentric laws can be coupled
so that $X_{k+1}\in\{X_k,x_{k+1}\}$ almost surely. For instance, in case
(ii), if $X_k=-a_k$ we set $X_{k+1}=x_{k+1}=-a_{k+1}$, whereas if
$X_k=b_k$ we keep $b_k$ with probability $
\tfrac{a_{k+1}(a_k+b_k)}
     {a_k(a_{k+1}+b_k)}\in(0,1)$.
A direct computation then gives $X_{k+1}\sim  \mu_{a_{k+1},b_{k+1}}$.
The case (iii) is analogous.

Let us now derive the costs corresponding to this strategy. First, note that $c_k \leq 0$ whenever $k$ falls in the case (i), by the preceding remark.

In the cases (ii) and (iii), let us first write, by definition
\[c_k = \E\left[\ln\frac{|X_k-x_{k+1}|}{|x_{k+1}|}\right] = \frac{b_k}{a_k+b_k} \ln \frac{|-a_k-x_{k+1}|}{|x_{k+1}|} +  \frac{a_k}{a_k+b_k} \ln\frac{|b_k-x_{k+1}|}{|x_{k+1}|}.\]
In the case (ii), we find $c_k = C(a_k, b_k;a_{k+1})$, while in the case (iii), we have $c_k = C(b_k, a_k; b_{k+1})$.
Since $\Phi$ is a potential, we have \[c_k =   C(a_k, b_k; a_{k+1}) \leq \Phi(a_k,b_k) - \Phi(a_{k+1},b_k) =   \Phi(a_k,b_k) - \Phi(a_{k+1},b_{k+1}),\] in the case (ii), and 
\[c_k = C(b_k, a_k; b_{k+1}) \leq \Phi(b_k, a_k) - \Phi(b_{k+1}, a_k) =\Phi(a_k, b_k) - \Phi(a_{k+1}, b_{k+1}).\] 
Concluding, we have $c_k \leq \Phi(a_k, b_k) - \Phi(a_{k+1}, b_{k+1})$ for all $k^\star \leq k \leq q-2$, since this inequality is trivially valid for $k$ in the case (i), owing to $(a_{k+1}, b_{k+1}) = (a_k, b_k)$.

Combining all the previous estimates, we find
\begin{align*}
\E[\ln Q_1(X)]&  \leq \E\left[\ln\frac{|X_{q-1}|}{x_0}\right] + \ln \frac{x_0}{b_{k^\star}}+  \ln \left(\frac{a_{k^\star} + b_{k^\star}}{a_{k^\star}}\right)+ \sum_{k=k^\star}^{q-2} c_k \\
& \leq \E[\ln |X_{q-1}| ] +   \ln \left(\frac{a_{k^\star} + b_{k^\star}}{a_{k^\star} b_{k^\star}}\right)+ \sum_{k=k^\star}^{q-2} \Phi(a_k, b_k) - \Phi(a_{k+1}, b_{k+1}) \\
& =  T(a_{q-1}, b_{q-1}) +    \ln \left(\frac{a_{k^\star} + b_{k^\star}}{a_{k^\star}b_{k^\star}}\right) + \Phi(a_{k^\star}, b_{k^\star}) - \Phi(a_{q-1}, b_{q-1}) \\
& = G_\Phi(a_{k^\star}, b_{k^\star}, a_{q-1}, b_{q-1}).
\end{align*}

It remains to check that $(a_{k^\star},b_{k^\star},a_{q-1},b_{q-1})\in E_R.
$
First, since $k^\star\ge1$, the separation assumption gives
\[
1\le a_{k^\star}=-x_{k^\star}\le R.
\]
Moreover, the sequence $(a_k)_{k\ge k^\star}$ is non-increasing.
Indeed, whenever it changes, we are in case~(ii), and $a_{k+1}=-x_{k+1}<e^{-\lambda}a_k<a_k$.
Since $k+1\le q-1$, the separation assumption also yields
$a_{k+1}=|x_{k+1}|\ge1$. Hence $
1\le a_{q-1}\le a_{k^\star}\le R$.

We now distinguish two cases for the positive endpoint.
Assume first that $x_0\ge1$. During the initial one-sided phase,
$X_k$ is always equal either to $x_0$ or to one of the points
$x_j$, $1\le j\le k$. Hence $
1\le b_{k^\star}=X_{k^\star-1}\le R$.

After time $k^\star$, the sequence $(b_k)$ is non-increasing, and
whenever it changes we are in case~(iii), so that
$b_{k+1}=x_{k+1}$. Since $k+1\le q-1$, we again have $x_{k+1}\ge1$. Therefore
$1\le b_{q-1}\le b_{k^\star}\le R$.

Assume now that $0<x_0<1$. Since every positive $x_j$ with
$1\le j\le q-1$ satisfies $x_j\ge1$, no change can occur during the
initial one-sided phase: indeed $x_j\ge1>e^{-\lambda}x_0$. Consequently, $b_{k^\star}=x_0<1$.
Moreover, case~(iii) can never occur after time $k^\star$. Indeed, as long as $b_k=x_0<1$, case~(iii) would require
$0<x_{k+1}<e^{-\lambda}b_k<1$, which contradicts $|x_{k+1}|\ge1$. Thus $b_{q-1}=b_{k^\star}=x_0<1$.

In both cases,
\[
1\le a_{q-1}\le a_{k^\star},
\qquad
\bigl(
b_{q-1}=b_{k^\star}<1
\ \text{or}\
1\le b_{q-1}\le b_{k^\star}
\bigr),
\]
and all four quantities belong to $(0,R]$. Hence $
(a_{k^\star},b_{k^\star},a_{q-1},b_{q-1})\in E_R$, and we may set $(a,b,a',b') := (a_{k^\star},b_{k^\star},a_{q-1},b_{q-1})$.

Finally, let us bound $m(X)$. 
If $b_{k^\star}\geq1$, the initial one-sided phase contains at most $
\lambda^{-1}\ln\tfrac{R}{b_{k^\star}}$ indices for which $X_k \neq X_{k+1}$, while after $k^\star$ the product $a_k b_k$ decreases by at least
a factor $e^\lambda$ whenever an update occurs, hence this second phase
contains at most
$\lambda^{-1}\ln(R\, b_{k^\star})$
updates. Their sum is at most $2\lambda^{-1}\ln R$. If
$b_{k^\star}<1$, there is no update in the initial phase and $b_k$ remains
constant afterwards, yielding the stronger bound $\lambda^{-1}\ln R$.
Accounting for the transition at $k^\star$ and for the $+1$ in the
definition of $m(X)$ gives
\[
m(X)\leq2+2\lambda^{-1}\ln R, \quad \text{a.s.}
\]
Consequently,
\[
\E[\ln Q_\tau(X)]
\le
G_\Phi(a,b,a',b')
+
\bigl(2+2\lambda^{-1}\ln R\bigr)\ln(\tau^{-1}),
\]
and the result follows.
\end{proof}
\begin{rmrk}
Estimating \(\mathbb E[m(X)]\) rather than bounding $m(X)$ almost surely leads to a better dependence of \(\gamma_\tau\) on \(\tau\). Since this refinement only affects the quasi-Leja case and requires keeping track of less transparent constant terms, we do not pursue it here.
\end{rmrk}

\subsection{Choice of a potential}

The goal is to find a potential \(\Phi\) satisfying~\eqref{telescope} for which the supremum of $G_\Phi$ over $E_R$ can be controlled linearly as a function of $\ln(R)$ with the smallest possible coefficient in front of $\ln(R)$.

Given the form of $G_\Phi$, it is convenient to define
\begin{equation}
\label{nice_potential}
\Phi_\gamma(a,b) = T(a,b) + \gamma \ln(a+b). 
\end{equation}
Indeed, if a potential is of the form~\eqref{nice_potential}, then we have 
\[\sup_{E_R} G_{\Phi_\gamma} \leq \ln(2) 
+  \gamma \ln(R).\]
This follows from the identity
\[G_{\Phi_\gamma}(a,b,a',b') = h\left(\frac{a}{a+b}\right) 
+ \gamma \ln \left(\frac{a+b}{a'+b'}\right) ,\]
where \[h : p \mapsto -p \ln(p)- (1-p) \ln(1-p)\] is the entropy, which satisfies $h(p) \leq \ln(2)$ for all $0 < p<1$.  On the other hand, if $b= b'<1$, then $\tfrac{a+b}{a'+b'} = \tfrac{a+b}{a'+b} \leq \tfrac{R+b}{1+b} \leq R$, while if $1 \leq b'\leq b$,  $\tfrac{a+b}{a'+b'} \leq \tfrac{2R}{2} = R$.

\begin{rmrk}
The value $\gamma = 9/8$ of~\cite{AN} is obtained by using 
\[\Phi(a,b) = T(a,b) + \gamma \ln(a+b) - h(\tfrac{a}{a+b})+ \tfrac{1}{8} \ln \min(\tfrac{a}{a+b},\tfrac{b}{a+b}).\]
\end{rmrk}

As a result, we may now focus on finding a function of the form~\eqref{nice_potential} that satisfies~\eqref{telescope}.
In this case, the inequality~\eqref{telescope} can be normalised so as to depend only on two variables. Indeed by homogeneity, one may set $b=1$ and with $p := \tfrac{a}{a+b}$, $q := \tfrac{c}{c+b}$ (with $0<q<p<1$), so that 
letting
\[(p,q) \mapsto \gamma \ln \left( \frac{1-q}{1-p}\right) +(1-p) \ln \left(\frac{p}{p-q} \right)- h(q) 
,\]
$\Phi_\gamma$ of the form~\eqref{nice_potential} satisfies~\eqref{telescope} if and only if the above function is nonnegative for all $0<q<p<1$.

The best such constant $\gamma^\star$ is then defined 
by \[\gamma^\star := \sup_{0<q<p<1} \frac{
- q \ln(q) -(1-q)\ln(1-q)-(1-p)\ln \left(\frac{p}{p-q}\right)
}{
\ln\!\left(\frac{1-q}{1-p}\right)}, \]
thus recovering and explaining the definition~\eqref{gamma_star}.

To conclude the proof of Lemma~\ref{improved_AN}, we only need to show the following Lemma, whose proof is postponed to the Appendix.
\begin{lmm}
\label{elem_lemma}
$\gamma^\star$ satisfies  $\gamma^\star < +\infty$.
\end{lmm}

We may now conclude the proof of Lemma~\ref{improved_AN}: since \(\Phi_{\gamma^\star}\) is a potential and \[\sup_{E_R}G_{\Phi_{\gamma^\star}}\le\ln2+\gamma^\star\ln R,\] we may use Lemma~\ref{lem:potential-reduction} to deduce Lemma~\ref{improved_AN}.
\begin{rmrk}
More general potentials involving an additional bounded symmetric function of \(a/(a+b)\) may lead to slightly smaller exponents. We do not investigate this additional optimisation here.
%
%
\end{rmrk} 


\section{Local improvement}
From now on, we assume that \(K\) is regular. Recall that $d_K = \mathrm{diam}(K)$ is its diameter, and $\|\cdot\|_K$ is the supremum norm of continuous functions from $K$ to $\C$. We refer to the book~\cite{SaffTotikBook1997} for the necessary material related to potential theory.

For $t >0$, we define 
\[\forall x \in K, \quad \rho_{t,K}(x) = \mathrm{dist}\big(x, \{g_K = t\}\big),\]
where $g_K$ is the Green function of $K$; we will mostly be using $\my$, for $n \in \N^*$. 
With the notations of~\cite{AN}, letting $G_K : \delta \mapsto \textstyle \max \{g_K(z)\, |\, \mathrm{dist}(z,K) \leq 2 \delta\}$, we have $\textstyle \inf_{x \in K} \rho_{t,K}(x) = 2 G_K^{-1}(t)$ for all $t>0$, where $G_K^{-1} : t \mapsto \inf \{\delta>0\, |\,G_K(\delta) \geq t\}$.\footnote{This generalised inverse happens to boil down to a genuine inverse for finite unions of intervals.}

Let us collect a few basic but useful results regarding this function: 
\begin{itemize}
\item as a distance function to a closed set, $\rho_{t,K}$ is a $1$-Lipschitz function for any $t>0$,
\item for two regular compact sets $K_1 \subset K_2$, the fact that the corresponding Green functions satisfy $g_{K_1} \geq g_{K_2}$ implies $\rho_{t,K_1}(x)  \leq \rho_{t,K_2}(x)$ for all $x \in K_1, \, t>0$,
\item for $n \in \N^*$, $\{g_K = \tfrac{1}{n}\} \subset \{g_K \leq 1\}$ so that $\rho_{1/n,K}(x)  \leq \mathrm{diam}(\{g_K \leq 1\})$ for all $x \in K$. Hence, the sequence $(\|\rho_{1/n,K}\|_K)_{n \in \N^*}$ is bounded. 
\end{itemize}
From now on, we write $\my$ instead of $\rho_{1/n,K}$ to ease notations.


Let us start with the estimate comparing the spacing $\e_{k,n}^\bx$ to the function $\my$ at the point~$x_k$. This result holds without any further assumption on $K$, which is only assumed to be a regular compact subset of $\C$.
\begin{lmm}
\label{est_spacing}
For all $0 \leq k \leq n-2$, letting $c_\tau:= \tfrac{\tau}{1+e}$, there holds
\[\forall \bx \in \leja, \quad \e_{k,n}^\bx \geq c_\tau \, \my(x_k).\]
\end{lmm}
\begin{proof} Let  \(0\le k<j\le n-1\) be fixed, and define $P_j(x):= \textstyle \prod_{r=0}^{j-1}(x-x_r)$. By the $\tau$-Leja property, we have $|P_j(x_j)|\geq \tau\|P_j\|_K$.

 Let $\rho:=\my(x_k)$, and choose any \(r\) with \(0<r<\rho\). 
By definition of \(\rho\), \[ \forall z \in D(x_k, r), \quad g_K(z)\le \tfrac1n.\]
 By the Bernstein--Walsh lemma, since $P_j$ has degree $j$, 
 \[ \forall z \in D(x_k, r), \quad |P_j(z)| \le \|P_j\|_K \, e^{j g_K(z)} \le  e^{j/n}  \|P_j\|_K \leq e  \|P_j\|_K. \] 
 For $\e:=|x_j-x_k|$, we now prove  $\e \ge \tfrac{\tau r}{\tau+e}$. Suppose, to the contrary, that $\e<\tfrac{\tau r}{\tau+e}$; in particular, $\e<r$. For every point $z \in [x_k, x_j]$, the disk \(D(z,r-\e)\) is contained in \(D(x_k,r)\). Hence Cauchy's estimate gives \[ |P_j'(z)| \le \frac{e \|P_j\|_K}{r-\e}, \] for every \(z\) on that segment. 
 Therefore, thanks to $P_j(x_k)=0$, we obtain \[ \tau  \|P_j\|_K \leq  |P_j(x_j)-P_j(x_k)| \le  \e \sup_{z\in[x_k,x_j]} |P_j'(z)| \le \e \,\frac{e  \|P_j\|_K}{r-\e}. \] Since \( \|P_j\|_K>0\), this implies $\tau \le  \tfrac{\e e}{r-\e}$, which contradicts the assumption. Hence the claim is proved. Since this holds for every \(0<r<\rho\), letting \(r\) tend to $\rho$ and using $\tau \leq 1$ yields \[ \e = |x_j-x_k| \ge \tfrac{\tau}{\tau+e} \, \my(x_k)\geq  \tfrac{\tau}{1+e}\,\my(x_k) \] This completes the proof. \end{proof}

We now turn to estimating the function $\rho_{t,K}$ from below for a finite union of intervals, while providing an explicit formula when $K = [-1,1]$.

\begin{lmm}
\label{estimate_rho}
Assume that $K$ is a finite union of intervals, $K = \cup_{j=1}^r [a_j, b_j]$ with $a_1 < b_1 < a_2 < \ldots < b_{r-1}< a_r < b_r$.
Then, there exists $c_j> 0$ such that
\begin{equation}
\label{lower_bound_rho}
\forall t>0, \; \forall x \in [a_j, b_j], \quad \rho_{t,K}(x) \geq c_j \max \Big(t \sqrt{\mathrm{dist}(x, \{a_j, b_j\})}, t^2 \Big).
\end{equation}

In the case where $K = [-1,1]$, we have the explicit formula
\begin{equation}
\label{explicit_rho}
\forall t > 0, \; \forall x \in [-1,1], \quad \rho_{t,K}(x) = \begin{cases}
\sinh(t) \sqrt{1-x^2}&  \text{ if } |x| < \cosh(t)^{-1}\\
\cosh(t) -|x|  & \text {else} 
\end{cases}.
\end{equation}
\end{lmm}

%

\begin{proof}
Let $t>0$ be fixed. 

\textit{Explicit computation for $K=[-1,1]$.}
We start with the case $K = [-1,1]$, for which it is well known that the Green function is given by $z \mapsto  \ln(|z + \sqrt{z^2-1}|)$ with the branch chosen so that $\sqrt{z^2-1} \sim z$ as $z \to \infty$~\cite[Example~3.6]{Saff2010}. Let $t>0$, and set $c_t:=\cosh t$, $s_t = \sinh t$.
The level line \(\{g_K=t\}\) is thus an ellipse, namely
\[
\{g_K=t\}
=
\left\{
c_t \cos\theta+i \, s_t\sin\theta \,| \, 0\le \theta<2\pi
\right\}.
\]
Let $x \in [-1,1]$; by symmetry we may assume \(0\le x\le1\). 
Putting \(u=\cos\theta\), we obtain
\[
\rho_{t,K}^2(x)
=
\min_{-1\le u\le 1} \;
(x-c_t u)^2+s_t^2(1-u^2).
\]
Since \(c_t^2-s_t^2=1\), the expression of interest reads
\[
(x- c_t u)^2+s_t^2(1-u^2)
=
u^2-2 c_t xu+x^2+s_t^2.
\]
If $c_t x \in [0,1)$, $u = c_t x$ is the unique critical point of this convex function. 
Hence, there are two cases. First, if \(c_t x<1\), the function is minimised at $u = c_t x$ (and not at $u = \pm 1$), which leads to $\rho_{t,K}^2(x) = s_t^2 (1-x^2)$. 

Second, if \(c_t x  \geq 1\), then the minimum is attained at \(u=1\), and a direct computation leads to the expected $\rho_{t,K}^2(x) = (c_t-x)^2$.

\textit{Lower bound~\eqref{lower_bound_rho} for $K=[-1,1]$.}
Using $s_t \geq t$, we find for $|x| < c_t^{-1}$, 
\[\rho_{t,K}(x) \geq t \sqrt{1-x^2}  \geq t \sqrt{1-|x|} = t \, \sqrt{\mathrm{dist}(x, \{-1, 1\})}.\]
Moreover, since $\sqrt{1-x^2} \geq \tanh(t)$, 
\[\rho_{t,K}(x) \geq s_t  \tanh(t) = c_t - c_t^{-1} \geq c_t - 1 \geq \tfrac{t^2}{2}.\]
 For $|x| \geq c_t^{-1}$, we have $\max (t \sqrt{\mathrm{dist}(x, \{-1,1\})}, t^2) = t^2$, and 
\[\rho_{t,K}(x) \geq c_t - 1 \geq \tfrac{t^2}{2},\]
so the desired estimate holds (with the constant $c= \tfrac{1}{2}$).
By affine rescaling, such an estimate then holds for any interval $[a,b]$ upon changing the constant.

\textit{General lower bound~\eqref{lower_bound_rho}.}
Now assume that $K = \cup_{j=1}^r [a_j, b_j]$, with $a_1 < b_1 < a_2 < \ldots < b_{r-1} < a_r < b_r$.

For $j \in \{1, \ldots, r\}$, since $[a_j,b_j] \subset K$, it follows that $\rho_{t,K}(x) \geq \rho_{t, [a_j,b_j]}(x)$ for all $x \in [a_j, b_j]$.
Hence, by the previous estimate we find for all such $x$ and some $c_j > 0$
\[\rho_{t,K}(x) \geq c_j \max \Big(t \sqrt{\mathrm{dist}(x, \{a_j, b_j\})}, t^2 \Big),\]
and the result is proved.
\end{proof}

Next, we move on to an abstract lemma that compares $\textstyle \sum_{k=0}^{n-1} \rho^{-\alpha}(x_k)$ to $\textstyle \int_K \rho^{-(\alpha+1)}(x) \,dx$ for all $\alpha >0$, for a given function $\rho$ that will eventually be taken to be $\my$.
\begin{lmm}
\label{sum_to_int}
Let $K \subset \R$ be a measurable set. 
Let $x_0, \ldots, x_{n-1} \in K$ and  $\rho : K \to (0,+\infty)$ be a $1$-Lipschitz function such that
\begin{itemize}
\item there exists $c>0$ such that for all $0 \leq k < j \leq n-1$, $|x_k -x_j| \geq c\,  \rho(x_k)$,
\item letting $\eta :=\tfrac{c}{2+c}$,  there exists $\e>0$ such that \[\forall k \in \{0, \ldots, n-1\},\quad
\left|
K\cap (x_k-\eta\rho(x_k),x_k+\eta\rho(x_k))
\right|
\ge \e\,\eta\,\rho(x_k).
\]
\end{itemize}
Then, for any $\alpha > 0$, putting $M := \tfrac{2+ c}{\e\, c} (\tfrac{2+2c}{2 + c})^{\alpha+1}$, there holds
\[\sum_{k=0}^{n-1} \rho^{- \alpha}(x_k) \leq M \int_K \rho^{-(\alpha+1)}(x)\,dx.\]
\end{lmm}
\begin{rmrk}
\label{rmk_eps}
If $K$ is a finite union of intervals, note that the existence of $\e$ is always granted by choosing it small enough. We shall later on apply this result to $\my$, while making sure that $\e$ can be chosen independently of $n$. 
\end{rmrk}
\begin{proof}
Set $\rho_k:=\rho(x_k)$, as well as the open intervals
\[
I_k := \bigl(x_k-\eta\rho_k,\ x_k+\eta\rho_k\bigr).
\]

Let us first show that the intervals \(I_0,\dots,I_{n-1}\) are pairwise disjoint.
Suppose that \(I_k\cap I_j\ne\varnothing\) for some \(k<j\). Since the
intervals are open, $|x_k-x_j|<\eta(\rho_k+\rho_j)$. By the \(1\)-Lipschitz property $\rho_j\le \rho_k+|x_k-x_j|$, hence $(1-\eta)\rho_j<(1+\eta)\rho_k$.
Therefore $\rho_j<\tfrac{1+\eta}{1-\eta}\rho_k$.
Substituting this into the overlap inequality gives
\[
|x_k-x_j|
<
\eta\big(1+\tfrac{1+\eta}{1-\eta}\big)\rho_k
=
\tfrac{2\eta}{1-\eta}\rho_k = c \rho_k,
\]
contradicting the assumed separation. Hence the intervals $I_0, \ldots, I_{n-1}$ are pairwise disjoint, and so are the intervals $K \cap I_0, \ldots, K \cap I_{n-1}$.

Now, for $x \in I_k$, $\rho(x) \leq  \rho_k + |x - x_k| \leq \rho_k +  \eta \rho_k$, which after exponentiating to the power $-(\alpha + 1)$ and integrating on $K \cap I_k$, leads to 
\[|K \cap I_k| \, \rho_k^{-(\alpha+1)} \leq (1+\eta)^{\alpha+1}\int_{K \cap I_k} \rho^{-(\alpha+1)}(x) \,dx. \]
Hence, using $|K \cap I_k| \geq  \e \eta \rho_k$, we have \[ \rho^{-\alpha} (x_k)= \rho_k^{-\alpha}  \leq (\e \,\eta)^{-1} (1+\eta)^{\alpha+1} \int_{K \cap I_k} \rho^{-(\alpha+1)}(x) \,dx.\]
Summing over $k \in \{0,\ldots, n-1\}$ and using the disjointness of the intervals $I_k$, the announced bound follows.
\end{proof}

We finish this series of Lemmas by estimating the integral of $\my^{-(\alpha+1)}$ over $K$.
\begin{lmm}
\label{estimate_int}
Assume that $K$ is a finite union of intervals, and let $\alpha > 0$.
Then, \[\int_K \my^{-(\alpha+1)}(x)\,dx =  \begin{cases} O(n^{1+\alpha}) & \text{ if } \alpha< 1 \\
O(n^2 \ln(n)) & \text{ if } \alpha = 1, \\
O(n^{2 \alpha}) & \text{ if } \alpha>1.
\end{cases}.\]
\end{lmm}
\begin{proof}
Write $K = \cup_{j=1}^r [a_j, b_j]$, with $r \geq 1$, $a_1 < b_1 < a_2 < \ldots < b_{r-1} < a_r < b_r$.
Fix $j \in \{1, \ldots, r\}$, we show that the integral over $[a_j,b_j]$ is $O(n^{1+\alpha})$ if $\alpha<1$, $O(n^2 \ln(n))$ if $\alpha =1$, and $O(n^{2\alpha})$ if $\alpha>1$. 

Using the estimate~\eqref{lower_bound_rho},
we have
\begin{align*}
\int_{a_j}^{b_j} \my^{-(\alpha+1)}(x)\,dx  & \leq c_j^{-(\alpha+1)} \int_{a_j}^{b_j}\max \Big(\tfrac{1}{n} \sqrt{\mathrm{dist}(x, \{a_j, b_j\})},\, \tfrac{1}{n^2} \Big)^{-(\alpha+1)}\,dx.
\end{align*}
The contribution of the latter integral near the endpoints is, for $n$ large enough and by symmetry
\begin{align*}
2  \int_{a_j}^{a_j + \tfrac{1}{n^2}}\max \Big(\tfrac{1}{n} \sqrt{\mathrm{dist}(x, \{a_j, b_j\})},\, \tfrac{1}{n^2} \Big)^{-(\alpha+1)}\,dx = 2  \int_{a_j}^{a_j + \tfrac{1}{n^2}} (\tfrac{1}{n^2})^{-(\alpha+1)}  \,dx = 2 n^{2 \alpha}.
\end{align*}
For $\alpha \neq 1$, the interior contribution reads, again for $n$ large enough and by symmetry,
\begin{align*}
2 ( \tfrac{1}{n})^{-(\alpha+1)} \int_{a_j + \tfrac{1}{n^2}}^{\frac{a_j+b_j}{2}}& \Big(\sqrt{\min(x-a_j, b_j -x)}\Big)^{-(\alpha+1)} \, dx  =2 n^{\alpha+1}  \int_{a_j + \tfrac{1}{n^2}}^{\tfrac{a_j+b_j}{2}} (x-a_j)^{-(\alpha+1)/2} \, dx  \\
& = 2  \tfrac{2}{\alpha-1} n^{\alpha+1} \Big((\tfrac{1}{n^2})^{\tfrac{1-\alpha}{2}} - (\tfrac{b_j-a_j}{2})^{\tfrac{1-\alpha}{2}} \Big).
\end{align*}
If $\alpha<1$, the constant term dominates, while if $\alpha>1$, the $n$-dependent term does. Hence, the interior contribution is $O(n^{\alpha + 1})$ if $\alpha<1$, and $O(n^{2\alpha})$ if $\alpha>1$. Finally, for $\alpha = 1$, the interior integral is bounded above by $2 n^2 \ln(\tfrac{b_j-a_j}{2} n^2)$ which is $O(n^2 \ln(n))$.

Comparing the interior and boundary contributions, we obtain the announced result.

\end{proof}

\bigskip

\noindent
\textit{Proof of Theorem~\ref{main_thm}.}

Let $\bx \in \leja$.  For $k = n-1$, the $\tau$-Leja property yields $\|\ell_{k,n}^\bx\|_K \leq \tau^{-1}$. 
Hence, we have by Lemmas~\ref{improved_AN} and~\ref{est_spacing}
\begin{align*} \Lambda_n^\bx & \leq \sum_{k=0}^{n-1} \|\ell_{k,n}^\bx\|_K \leq 2 \tau^{-2} d_K^{\gamma_\tau}  \sum_{k=0}^{n-2} (\e_{k,n}^\bx)^{-\gamma_\tau} +  \tau^{-1} \\ 
& \leq  2 \tau^{-2} d_K^{\gamma_\tau}\, c_\tau^{-\gamma_\tau}  \sum_{k=0}^{n-2} \my^{-\gamma_\tau}(x_k) + \tau^{-1}.
\end{align*}
Let us show that $\my$ satisfies the hypotheses of Lemma~\ref{sum_to_int}. It is a $1$-Lipschitz function, and by Lemma~\ref{est_spacing}, it satisfies the first hypothesis (i), with $c = c_\tau$. As discussed in Remark~\ref{rmk_eps}, the existence of $\e>0$ such that (ii) holds is true, but we need to make sure that $\e$ can be taken independently of $n$.

We write $K = \cup_{j=1}^r [a_j, b_j]$ with $a_1<b_1 < a_2< \ldots < b_{r-1} < a_r < b_r$.  For $k \in \{0,\ldots, n-1\}$, recalling the notation $\eta=\tfrac{c_\tau}{2+c_\tau}$, we have for $x_k \in [a_j,b_j]$
\[
\left|K\cap (x_k-\eta \my(x_k),x_k+\eta\my(x_k))
\right|
\ge \min(\eta \my(x_k), b_j - a_j)
\]
We must  find $\e>0$ such that $ \min(\eta\, \my(x_k), b_j - a_j) \geq \e\, \eta\, \my(x_k)$ for all $n \in \N^*$ and $k\in \{0, \ldots, n-1\}$. Such an $\e>0$ exists because $(\|\my\|_K)_{n \in \N^*}$ is bounded. Consequently, returning to the preceding estimates, Lemma~\ref{sum_to_int} entails, setting $M_\tau(\e) := \tfrac{2+ c_\tau}{\e\, c_\tau} (\tfrac{2+2c_\tau}{2 + c_\tau})^{\gamma_\tau+1}$
\begin{align*} \Lambda_n^\bx &  \leq   \tau^{-1} +  2 \tau^{-2} d_K^{\gamma_\tau}\, c_\tau^{-\gamma_\tau} M_\tau(\e)  \int_K \my^{-(\gamma_\tau+1)}(x)\,dx  
\end{align*}
Lemma~\ref{estimate_int} then concludes the proof.
\qed

%
%

\paragraph{Limitations and perspectives.}

First, since our approach is grounded in the estimate~\eqref{bound_AN_eq} of~\cite{AN} which crucially requires $K \subset \R$, an extension to more general regular compact sets $K \subset \C$ requires significantly new ideas.

Establishing an improved polynomial estimate for regular compact subsets of \(\R\) beyond finite unions of intervals might be possible. This would require generalising the approximation of the sum by an integral with respect to a suitably defined measure; the equilibrium measure of \(K\) is a natural candidate. We have not been able to do so.

We end up with an estimate $\Lambda_n^\bx = O(n^{1+\gamma^\star})$ with $1+\gamma^\star \approx 1.6304$, for $\tau=1$. How far can we go with the current approach, and how close can we get to an $O(n)$ bound for $K = [-1,1]$? 
The present approach also seems to have a natural limitation. Indeed, even if the exponent in the Key Lemma were further improved, the example discussed above shows that no universal estimate of this type can hold with an exponent smaller than $\log_2(3/2)\approx 0.5850$. Such an exponent would still lead, within the present local-separation and packing argument, to a bound of order
\[
n^{1+\log_2(3/2)} \approx n^{1.5850}.
\]
Thus, our exponent $1+\gamma^\star\approx1.6304$ is already fairly close to the apparent limit of the present framework. Any substantial progress towards the near-linear growth observed numerically would therefore seem to require additional information on the specific structure of Leja sequences, rather than treating the relevant tails merely as arbitrary separated configurations of real points.




\paragraph{Declaration of generative AI use.}
During the development and preparation of this work, the author used ChatGPT for language editing, proofreading and as an exploratory aid in developing technical estimates. In particular, starting from the author's observation that the global $n^{-2}$ separation of Leja points on $[-1,1]$ conceals finer local separation away from the endpoints, AI-assisted discussions helped refine and turn this intuition into a quantitative local-separation approach. They also contributed to identifying the sharpening of the exponent within the Key Lemma of~\cite{AN}, leading to the optimised constant $\gamma^\star$. All mathematical statements, proofs, computations, and references appearing in the manuscript were independently checked and validated by the author, who takes full responsibility for the content of the work.

\bibliographystyle{unsrt}
\bibliography{Biblio_Leja.bib}

\appendix

\section{Proof of Lemma~\ref{elem_lemma}}

Set $t:=\tfrac{p-q}{1-p}$, so that
$\ln\tfrac{1-q}{1-p}=\ln(1+t)$.
Away from $t=0$, the claim follows immediately from
\[
h(q)-(1-p)\ln \frac{p}{p-q}\le h(q)\le\ln2.
\]
As $t\to0$, the numerator can be written as
\[
h(q)
-\frac{1-q}{1+t}
 \ln\frac{q+t}{t(1-q)}.
\]
If $q$ stays in a compact subset of $(0,1)$, this tends to $-\infty$.
It remains to consider the two corners. Set
\[
N(t,q)
:=
h(q)-\frac{1-q}{1+t}
\ln\frac{q+t}{t(1-q)}.
\]
We claim that $N(t,q)_+=O(t)$ as $t\to0$,
uniformly near both $q=0$ and $q=1$.

Near $q=0$, using $h(q)\le q\ln(e/q)$, we have, for $q$ and $t$
sufficiently small,
\[
N(t,q)
\le
q\ln\frac eq-\frac12\ln\left(1+\frac qt\right).
\]
If $q\ge t$, the right-hand side is negative provided $q$ is restricted
to a sufficiently small fixed neighbourhood of $0$. If $q<t$, write
$q=tz$, $0<z<1$. Since $\ln(1+z)\ge z/2$,
\[
N(t,tz)
\le
z\left[t\left(1+\ln\frac1t-\ln z\right)-\frac14\right].
\]
For $t$ sufficiently small, $t(1+\ln(1/t))\le1/4$, and hence
\[
N(t,tz)\le tz(-\ln z)\le \frac te.
\]

Near $q=1$, writing $s=1-q$ and using $h(s)\le s\ln(e/s)$ gives
\[
N(t,1-s)
\le
\frac{s}{1+t}
\left(1+t+\ln2+\ln t-t\ln s\right).
\]
For $t$ sufficiently small, $1+t+\ln2+\ln t\le0$. Thus, setting
$x=-\ln s$,
\[
N(t,1-s)_+
\le t x e^{-x}\le\frac te.
\]
Therefore $N(t,q)_+=O(t)$ uniformly at both corners. Since
$\ln(1+t)\sim t$, the quotient remains bounded from above.
\qed

\end{document}